\documentclass[letterpaper, 10 pt, conference]{ieeeconf} 
\IEEEoverridecommandlockouts

\usepackage[T1]{fontenc}
\usepackage{lmodern}

\usepackage{amsmath,amssymb,amsfonts}
\newtheorem{remark}{Remark}
\usepackage{algorithmic}
\usepackage{graphicx}
\usepackage{textcomp}
\usepackage{xcolor}
\usepackage{booktabs}
\usepackage{hyperref}
\usepackage{comment}
\usepackage{mathrsfs}
\usepackage{tikz-cd}

\usepackage{subcaption}
\usepackage{caption}
\newcommand{\R}{\mathcal{R}}
\renewcommand{\S}{\mathcal{S}}
\newcommand{\A}{\mathcal{A}}
\newcommand{\T}{\mathcal{T}}

\newtheorem{theorem}{Theorem}
\newtheorem{lemma}{Lemma}
\newtheorem{corollary}{Corollary}
\newtheorem{proposition}{Proposition}

\begin{document}

\title{\bf \Large
Linking PageRank, Time Reversal, and Policy Evaluation
}


\author{Konstantin Avrachenkov, Lorenzo Gregoris, Nelly Litvak
\thanks{This work was supported by a grant from NSP SmartProfile and is accepted at IEEE CDC 2026.}
\thanks{K. Avrachenkov is with Inria Sophia Antipolis, France
        {\tt\small k.avrachenkov@inria.fr}}%
\thanks{L. Gregoris is with Department of Mathematics and Computer Science, Eindhoven University of Technology, The Netherlands
        {\tt\small l.gregoris@tue.nl}}%
\thanks{N. Litvak is with Department of Mathematics and Computer Science, Eindhoven University of Technology, The Netherlands
        {\tt\small n.v.litvak@tue.nl}}%
}

\maketitle

\begin{abstract}
We establish a connection between policy evaluation in Markov decision processes and PageRank in network analysis.
For a fixed policy, we show that the value function of a discounted Markov
decision process can be obtained, up to an explicit rescaling, from the PageRank
vector of a suitably defined time-reversed Markov chain.
In this correspondence, the discount factor plays the role of the teleportation
parameter, while rewards induce the restart distribution.
Beyond the irreducible case, invoking quasi-stationary distributions and Doob $h$-transforms, 
we prove a general decomposition theorem showing
that policy evaluation for arbitrary finite MDPs reduces to a collection of PageRank problems on the recurrent and transient components of the
policy-induced Markov chain.
This framework naturally extends to undiscounted MDPs with terminal states and to transition-dependent rewards.
We conclude by showing efficiency of our approach on a numerical example of a sticky random walk on large deterministic and random graphs.
\end{abstract}


\section{Introduction}

Policy evaluation is an essential computational subroutine in dynamic programming methods such as policy iteration and approximate value-function methods. 
In Markov Decision Processes (MDPs), given a fixed policy, one seeks to compute its value function, i.e., the expected cumulative (discounted) reward starting from a given state \cite{puterman2014markov}.  
At its core, policy evaluation amounts to solving a linear system of Bellman equations, or equivalently a (discounted) Poisson equation for a Markov chain induced by the policy.

Despite its apparent simplicity, policy evaluation remains a computational bottleneck in large-scale problems.
The state space may be massive, the transition matrix sparse and highly non-symmetric.
A large body of work has therefore focused on accelerating Bellman solvers through asynchronous updates, Gauss--Seidel schemes, prioritized sweeping, and residual-based heuristics \cite{bertsekas1995generic,puterman2014markov,moore1993prioritized,andre1997generalized,kolobov2012planning}.


In parallel, the problem of computing stationary distributions of Markov chains has been studied extensively in network science and numerical linear algebra, most notably in the context of PageRank \cite{brin1998anatomy}.  
PageRank interprets importance scores of nodes as the stationary distribution of a random walk with random restart, and has led to a rich variety of scalable algorithms tailored to massive graphs, including power iteration variants, coordinate updates \cite{OPIC, mcsherry}, Monte--Carlo methods \cite{avrachenkov2007monte}, and sophisticated residual-based solvers~\cite{fan_chung}.  
Among recent residual-based solvers, Red-Light-Green-Light (RLGL) algorithm \cite{RLGL, RLGL_VAR} and {\sc FwPushSOR}~\cite{chen2023accelerating} often outperform classical approaches. 

This paper has the following contributions. First, for irreducible MDPs, we establish that policy evaluation can be reduced to computing PageRank of the time-reversed policy-induced Markov chain. 
In this correspondence, the discount factor plays the role of the teleportation parameter (damping factor) in the context of PageRank, while the rewards determine the restart distribution.    
Second, we prove a general decomposition theorem showing that policy evaluation
for arbitrary finite MDPs reduces to a collection of PageRank problems on
recurrent and transient components, characterized respectively by stationary
and quasi-stationary distributions. 
Third, we extend our framework to undiscounted MDPs with terminal states and to transition-dependent rewards by invoking quasi-stationary distributions and Doob $h$-transforms. 
Finally, we illustrate our theoretical results by solving a sticky random walk control problem on large deterministic and random graphs.

Taken together, our results suggest that policy evaluation can be fruitfully re-examined through the lens of PageRank and stationary distribution computation, leading to both theoretical insight and possibly practical performance gains.



The remainder of the paper is organized as follows.
Section~\ref{sec:2} reviews background material on policy evaluation, Markov
chains, and PageRank.
Section~\ref{sec:3} establishes the PageRank representation of the value function
for irreducible MDPs and presents the general decomposition theorem that
reduces policy evaluation to PageRank subproblems.
Section~\ref{sec:experiments} presents numerical experiments for the sticky random walk control problem.
Section~\ref{sec:conclusion} concludes the paper.

\section{Preliminaries}
\label{sec:2}

\subsection{Markov Chains and Stationary Distributions}

A finite, discrete-time Markov chain on the state space $\mathcal{S}$ is specified
by a row-stochastic matrix $P\in\mathbb{R}^{|\mathcal{S}|\times|\mathcal{S}|}$,
where $P(s,s')$ denotes the probability of transitioning from state $s$ to state
$s'$. 
If $P$ is irreducible, by the Perron–Frobenius theorem, there exists a unique stationary distribution
$\mu\in\mathbb{R}^{|\mathcal{S}|}$ satisfying
$$
\mu = \mu P,
\qquad
\sum_{s\in\mathcal{S}} \mu(s) = 1,
\qquad
\mu(s) > 0 \quad \forall s\in\mathcal{S}.
$$

Let $M := \mathrm{diag}(\mu)$.
The \emph{time reversal} of the Markov chain is defined as the chain with
transition matrix
\begin{equation}
\label{eq:time_reversal}
P^* := M^{-1} P^\top M .
\end{equation}
The matrix $P^*$ is row-stochastic, irreducible, and admits $\mu$ as its unique
stationary distribution.
Moreover, $P$ and $P^*$ share the same spectrum.

Time reversal plays a central role in what follows, as it allows us to convert
right-eigenvector equations associated with value functions into left-eigenvector
equations defining stationary distributions.

\subsection{Policy Evaluation}

We consider a Markov decision process (MDP) with finite state space $\mathcal{S}$,
finite action space $\mathscr{A}$, transition probabilities
$p(s' \mid s,a)$, and reward function $r(s,a)$, following standard notation \cite{puterman2014markov,barto2021reinforcement}.  
A (stationary, deterministic) policy $\pi$ is a mapping
$\pi:\mathcal{S}\to\mathscr{A}$.
Fixing a policy $\pi$ reduces the MDP to a Markov chain on $\mathcal{S}$ with
transition matrix
$$
P_\pi(s,s') := p(s' \mid s, \pi(s)),
$$
and reward vector
$$
r_\pi(s) := r(s,\pi(s)).
$$
We consider expected discounted rewards with discount factor $\gamma\in(0,1)$.
The central object of interest in \emph{policy evaluation} \cite{puterman2014markov} is the value function
\begin{equation}
\label{eq:vf_def}
v_\pi(s)
\;:=\;
\mathbb{E}\left[
\sum_{t=0}^{\infty} \gamma^t r_\pi(X_t)
\,\middle|\,
X_0=s
\right],
\end{equation}
representing the expected discounted return obtained when starting from state
$s$ and following policy $\pi$.
In vector form, $v_\pi\in\mathbb{R}^{|\mathcal{S}|}$ is the unique solution of the
Bellman equation
\begin{equation}
\label{eq:bellman}
v_\pi \;=\; \gamma P_\pi v_\pi + r_\pi .
\end{equation}

Since the value function depends linearly on the reward function, we may,
without loss of generality, restrict attention to non-negative rewards.


Equation~\eqref{eq:bellman} is a linear system that can be written equivalently as
a discounted Poisson equation
\begin{equation}
\label{eq:poisson}
(I - \gamma P_\pi) v_\pi = r_\pi .
\end{equation}
Since $P_\pi$ is row-stochastic and $\gamma<1$, the matrix
$I-\gamma P_\pi$ is invertible, and the solution admits the convergent
Neumann-series representation
\[
v_\pi
\;=\;
(I-\gamma P_\pi)^{-1} r_\pi
\;=\;
\sum_{t=0}^{\infty} \gamma^t P_\pi^t r_\pi .
\]
This representation makes explicit the interpretation of $v_\pi$ as the
accumulated discounted expected reward along trajectories of the Markov chain induced by
$\pi$.

Computing or approximating the solution of \eqref{eq:bellman} is the basic task
of policy evaluation and constitutes a fundamental subroutine in dynamic
programming, policy iteration, and approximate dynamic programming methods. 

The definition \eqref{eq:vf_def} also extends naturally to the undiscounted case
($\gamma=1$), provided the Markov chain induced by $\pi$ contains
\emph{terminal (absorbing) states}.
Let $\A\subset\mathcal{S}$ denote the set of terminal states, and assume
that for all $s\in \A$,
$$
P_\pi(s,s)=1,
\qquad
r_\pi(s)=0.
$$
We further assume that, starting from any initial state $s\in\mathcal{S}$, the
chain reaches $\A$ almost surely in finite time.
Under this assumption, Equation~\eqref{eq:bellman} is well defined even for $\gamma=1$.
Unlike the discounted case, the matrix $I-P_\pi$ is singular, since
$P_\pi$ has eigenvalue $1$ associated with the absorbing states.
However, when the rewards are non-negative and terminal states are reached almost
surely, the system admits a unique solution satisfying $v_\pi(s)=0$ for all
$\A$ (which can be viewed as stopping the process once it reaches a terminal state).
In this case, $v_\pi(s)$ can be interpreted as the expected total reward acquired before absorption.



\subsection{PageRank}

PageRank, originally introduced in \cite{brin1998anatomy}, is the stationary
distribution of a Markov chain obtained by augmenting a base transition matrix
with random restarts.
Let $P$ be a row-stochastic matrix on $\mathcal{S}$.
Fix a teleportation parameter $c\in[0,1]$ and a restart distribution
$u\in\mathbb{R}^{|\mathcal{S}|}$. 

In the original formulation of PageRank $u$ is the uniform distribution. Another common choice in the literature is $u = \delta_s$, i.e., the process restarts at a fixed state $s$, commonly known as \emph{personalized} PageRank \cite{jeh2003scaling}.

The PageRank vector $w$ is defined as the unique probability vector satisfying
the stationary equation
\begin{equation}
\label{eq:pagerank}
w = c\, w P + (1-c)\, u .
\end{equation}
Equivalently, $w$ is the stationary distribution of the Markov chain that, at
each step,
\begin{itemize}
    \item follows $P$ with probability $c$;
    \item 
    restarts with probability
    $1-c$ from a state sampled according to $u$.
\end{itemize}

If $c<1$ PageRank is well defined and unique. Like for the value function, the PageRank equation admits a Neumann-series representation
$$
w = (1-c)\;u (I-c P)^{-1} =(1-c)\;u\sum_{t=0}^\infty c^t P^t.
$$
This representation shows that the PageRank of a state can be viewed as a weighted sum over all paths leading to that state, with weights that decay exponentially with path length.

\section{Policy Evaluation as PageRank Problem}
\label{sec:3}

In this section we establish a correspondence between discounted policy
evaluation and PageRank.
We begin with the case where the transition matrix $P_\pi$ induced by policy $\pi$ is irreducible, and show that the value function can be obtained from PageRank computation on a suitably defined time-reversed
Markov chain.
We then show that this representation is consistent with the undiscounted--reward limit. Finally we extend this result to general reducible Markov decision
processes.

\subsection{Irreducible discounted case}

We first present a concise algebraic proof of the correspondence between PageRank and the value function, and then give a
probabilistic derivation that clarifies its interpretation.

\begin{theorem}
\label{thm:value_pagerank}
Assume that the transition matrix $P_\pi$ induced by a fixed policy $\pi$ is irreducible, and let $\mu$ be its unique stationary distribution.
Let $v_\pi$ denote the value function with discount factor $\gamma\in[0,1)$ and reward vector $r\neq 0$ with non-negative entries.
Define the restart distribution
$$
u(s) = \frac{r(s)\mu(s)}{\langle r \rangle_\mu},
\qquad
\langle r\rangle_\mu := \sum_{x\in\mathcal S} r(x)\mu(x).
$$
Let $w$ be the PageRank vector associated with the transition matrix $P^*_\pi$ as defined in \eqref{eq:time_reversal}, teleportation parameter $\gamma$, and restart distribution $u$.

Then the value function satisfies
\begin{equation}
\label{eq:vf_pr}
v_\pi(s)
=
\frac{\langle r \rangle_\mu}{1-\gamma}\,
\frac{w(s)}{\mu(s)},
\qquad \forall s\in\mathcal S .
\end{equation}
\end{theorem}

\noindent {\bf Proof:}
For brevity, we omit the policy subscript and write $P$ and $v$ instead of $P_\pi$ and $v_\pi$.
The value function satisfies the Bellman equation
$$
v = r + \gamma P v .
$$
Taking transposes and right-multiplying by $M = \mathrm{diag}(\mu)$ yields
$$
v^\top M=\gamma v^\top P^\top M + r^\top M.
$$
Using the definition of the time-reversed chain $P^* = M^{-1} P^\top M$, this can be rewritten as
$$
v^\top M = \gamma v^\top M P^* + r^\top M .
$$
Multiplying both sides by $(1-\gamma)/\langle r\rangle_\mu$ gives
\begin{align}
\frac{1-\gamma}{\langle r\rangle_\mu} v^\top M
=
\gamma \frac{1-\gamma}{\langle r\rangle_\mu} v^\top M P^*
+
(1-\gamma)\frac{r^\top M}{\langle r\rangle_\mu} .
\label{eq:big}
\end{align}
Define
\begin{equation}
\label{eq:v-to-w}
w := \frac{1-\gamma}{\langle r\rangle_\mu} v^\top M .
\end{equation}
Since $(r^\top M)/\langle r\rangle_\mu = u$, then Equation~\eqref{eq:big} becomes
$$
w = \gamma w P^* + (1-\gamma)u,
$$
which is precisely the PageRank equation. \hfill $\Box$

\begin{remark}
We note that the proportionality constant $\frac{\langle r \rangle_\mu}{1-\gamma}$ in Equation~\eqref{eq:vf_pr} is the expected discounted reward when the chain is initialized at stationarity.
\end{remark}


Theorem~\ref{thm:value_pagerank} shows that discounted policy evaluation on an
irreducible Markov chain is equivalent to computing the PageRank vector of the
time-reversed chain.
In this correspondence, the discount factor plays the role of the teleportation
parameter, while the reward function determines the restart distribution.

The PageRank representation of Theorem~\ref{thm:value_pagerank} is consistent
with how MDP with discounted reward is related to MDP with time average reward as $\gamma \to 1^-$.

\begin{corollary}[\cite{puterman2014markov}, Corollary~8.2.5]
\label{lem:undiscounted_limit}
Let $P_\pi$ be irreducible with stationary distribution $\mu$, and let $r$ be a
fixed reward vector. Then
$$
\lim_{\gamma \to 1^-} (1-\gamma)\, v_\pi(s)
=
\langle r \rangle_\mu ,
\qquad \forall s \in \mathcal S .
$$
\end{corollary}

Corollary~\ref{lem:undiscounted_limit} is usually proved via Laurent expansion of the resolvent. It follows directly from Theorem~\ref{thm:value_pagerank}, since $w(s) \to \mu(s)$ when $\gamma \to 1^-$, i.e., PageRank with no teleportation is simply the stationary distribution.

\medskip

We next give a probabilistic proof of Theorem~\ref{thm:value_pagerank}, which clarifies
the roles of time reversal and of the stationary distribution. This is essentially a per-component proof.

\noindent {\bf Proof:}[Probabilistic Proof of Theorem~\ref{thm:value_pagerank}]

Recall that the PageRank vector $w$ admits
the following path-wise representation. Let $J$ be a geometric random variable
with parameter $1-\gamma$, and let
$(X_t)_{t\ge0}$ be a Markov chain with transition matrix $P^*$ and initial
distribution $X_0\sim u$. Denote by $\mathbb{P}^*(\cdot)$ the law of this time-reversed chain. Then
$$
w(s) = \mathbb{P}^*(X_J=s)
= (1-\gamma)\sum_{t\ge0}\gamma^t\,\mathbb{P}^*(X_t=s).
$$

We now express this quantity in terms of the forward chain $P$. By the
definition of time reversal and using the telescoping product, we have for all $x,s\in\mathcal S$ and $t\ge0$,
\begin{align}
\label{eq:telescopic}
\mathbb{P}^*(X_t=s\mid X_0=x)
=
\frac{\mu(s)}{\mu(x)}\,\mathbb{P}(X_t=x\mid X_0=s),
\end{align}
where the probabilities on the right-hand side refer to the forward chain
$P_\pi$.

Using Equation~\eqref{eq:telescopic} and the definition of $u$,
\begin{align*}
w(s)
&=
(1-\gamma)\sum_{t\ge 0}\gamma^t
\sum_{x\in\mathcal S}
u(x)\,\mathbb{P}^*(X_t=s| X_0=x)
\\
&=
(1-\gamma)\sum_{t\ge 0}\gamma^t
\sum_{x\in\mathcal S}
\frac{r(x)\mu(x)}{\langle r\rangle_\mu}
\frac{\mu(s)}{\mu(x)}
\mathbb{P}(X_t=x| X_0=s)
\\
&=
\frac{\mu(s) (1-\gamma)}{\langle r\rangle_\mu}
\sum_{t\ge 0}\gamma^t
\sum_{x\in\mathcal S}
r(x)\,\mathbb{P}(X_t=x| X_0=s)
\\
&=
\frac{\mu(s) (1-\gamma)}{\langle r\rangle_\mu}
\sum_{t\ge 0}\gamma^t
\mathbb{E}(r(X_t)| X_0=s)
\\
&=
\frac{\mu(s)(1-\gamma)}{\langle r\rangle_\mu}\, v(s). \hspace{4.57cm}\Box
\end{align*}

The need for time reversal and the appearance of the stationary distribution manifest in the second equality. 

The intuition of Theorem~\ref{thm:value_pagerank} is as follows. The analogy between the discount factor and the teleportation parameter $\gamma$ is immediate because in both cases, the paths of length $k$ contribute with coefficient $\gamma^k$. The duality between the value function and the PageRank is the one between looking forward and backwards. In the value function, we start with state $s$ and ask ourselves how much reward we collect starting from $s$. In case of PageRank, we ask an equivalent question: in the long-run, how much reward different states contribute to state $s$? This reward equals $\mu(s)v(s)$. Each state $s'$ is visited with its corresponding stationary probability, thus, the reward it contributes to other states is proportional to $\mu(s')r(s')$. State $s'$  contributes this reward to the states, that have path of length $k$ from $s'$, with discount $\gamma^k$. Since the path is {to} $s$, we trace it back using the time-reversed matrix $P^*$. 

\subsection{Duality}

The correspondence established in
Theorem~\ref{thm:value_pagerank} can be understood as a
duality between linear operators associated with a fixed
policy. This duality arises from the canonical pairing
between functions and measures, and from the adjointness
of $P$ and its time reversal $P^*$; see, e.g., \cite[Section~1.6]{levin2017markov}.

Let $L^2(\S,\mu)$ denote the space of real-valued
functions on $\S$ equipped with the inner product
\[
\langle f,g\rangle_\mu
:=\sum_{s\in\S}f(s)g(s)\mu(s),
\]
and let $\mathcal M(\S,\mu)$ denote the space of signed
measures on $\S$ absolutely continuous with respect to
$\mu$. Functions are represented as column vectors and
measures as row vectors. These spaces are canonically
paired via
\[
\langle f,\nu\rangle:=\sum_{s\in\S}f(s)\nu(s).
\]
The time-reversed kernel $P^*$ is
the adjoint of $P$ in $L^2(\S,\mu)$.

For $\gamma\in[0,1)$, define the forward and backward
resolvents on $L^2(\S,\mu)$ by
\[
\begin{aligned}
\mathcal R_\gamma f&:=(I-\gamma P)^{-1}f,\qquad
\mathcal R_\gamma^*f&:=(I-\gamma P^*)^{-1}f.
\end{aligned}
\]
By a mild abuse of notation, we also use
$\mathcal R_\gamma^*$ for the dual action of the
backward resolvent on signed measures:
\[
\mathcal R_\gamma^*\nu:=\nu(I-\gamma P^*)^{-1}.
\]
The two actions satisfy
\[
\langle g,\mathcal R_\gamma^*\nu\rangle
=
\langle\mathcal R_\gamma^*g,\nu\rangle.
\]
Applying the forward resolvent to a reward function $r$
yields $v_\pi=\mathcal R_\gamma r$. Applying the backward
resolvent to a restart measure $u$ yields the PageRank
vector $w=(1-\gamma)\mathcal R_\gamma^*u$.

\begin{proposition}
\label{prop:resolvent_adjointness}
The forward and backward resolvents on $L^2(\S,\mu)$
are adjoint:
\[
\langle\mathcal R_\gamma f,g\rangle_\mu
=
\langle f,\mathcal R_\gamma^*g\rangle_\mu,
\qquad f,g\in L^2(\S,\mu).
\]
\end{proposition}

\begin{proof}
Since $P^*$ is the adjoint of $P$ in $L^2(\S,\mu)$,
the adjoint of $I-\gamma P$ is $I-\gamma P^*$.
The adjoint of the inverse is the inverse of the adjoint.
\end{proof}

The stationary distribution $\mu$ induces an
identification between functions and signed measures.
Define
\[
\varphi:L^2(\S,\mu)\to\mathcal M(\S,\mu),
\qquad
(\varphi f)(s):=f(s)\mu(s).
\]
This realizes the Riesz identification of $L^2(\S,\mu)$ with its dual;
see, e.g., \cite[Chapter~4]{rudin1991functional}.

\begin{theorem}
\label{thm:resolvent_conjugacy}
For every $f\in L^2(\S,\mu)$,
\[
\mathcal R_\gamma^*(\varphi f)
=
\varphi(\mathcal R_\gamma f).
\]
Equivalently, the following diagram commutes:
\[
\begin{tikzcd}
L^2(\S,\mu)
    \arrow[r, "\mathcal R_\gamma"]
    \arrow[d, "\varphi"']
&
L^2(\S,\mu)
    \arrow[d, "\varphi"]
\\
\mathcal M(\S,\mu)
    \arrow[r, "\mathcal R_\gamma^*"']
&
\mathcal M(\S,\mu)
\end{tikzcd}
\]
\end{theorem}

\begin{proof}
Let $f,g\in L^2(\S,\mu)$. Using the dual action on
measures, the definition of $\varphi$, and
Proposition~\ref{prop:resolvent_adjointness}, we obtain
\[
\begin{aligned}
\langle g,\mathcal R_\gamma^*(\varphi f)\rangle
&=\langle\mathcal R_\gamma^*g,\varphi f\rangle\\
&=\langle f,\mathcal R_\gamma^*g\rangle_\mu\\
&=\langle\mathcal R_\gamma f,g\rangle_\mu\\
&=\langle g,\varphi(\mathcal R_\gamma f)\rangle.
\end{aligned}
\]
Since this holds for every $g$, taking
$g=\mathbf 1_{\{s\}}$ for each $s\in\S$ shows
that the two measures are equal.
\end{proof}

\subsection{From irreducible to general MDPs}

It is interesting to go beyond the irreducibility assumption of Theorem~\ref{thm:value_pagerank}, because, in general, the Markov chain induced by a fixed policy may decompose into transient and recurrent communicating classes.
Nevertheless, we will show how the PageRank interpretation remains valid at the level of each class. The main difficulty arises in the treatment of transient classes, for which time reversal is not applicable. 


Let $\mathcal S$ be the state space and decompose it as
\[
\mathcal S = \T_1 \cup \cdots \cup \T_m \;\cup\; \R_1 \cup \cdots \cup \R_\ell ,
\]
where $\T_j$ are transient classes and $\R_k$ are recurrent classes.
After a suitable reordering of states, the transition matrix $P_\pi$ has the
canonical block upper–triangular form
\begin{equation}
\label{eq:matrix}
P_\pi =
\begin{pmatrix}
Q^{(1)} &   P_{\T_1\T_2}    &   \cdots     & P_{\T_1 \R_1} & \cdots & P_{\T_1 \R_\ell} \\
        & \ddots &        & \vdots      &        & \vdots        \\
        &        & Q^{(m)}& P_{\T_m \R_1} & \cdots & P_{\T_m \R_\ell} \\
\hline
0       & \cdots & 0      & P^{(1)}     &        & 0             \\
\vdots  &        & \vdots &             & \ddots &               \\
0       & \cdots & 0      & 0           &        & P^{(\ell)}
\end{pmatrix}.
\end{equation}
Here each $Q^{(j)}$ is a substochastic irreducible matrix on $\T_j$, and each
$P^{(k)}$ is a stochastic irreducible matrix on $\R_k$.

It will be convenient to consider the restriction of the value and reward vectors on each class
$$
v^{(k)} := 
\begin{cases}
    v\mid_{\T_k} \quad \text{ for } k=1,\dots,m \\ 
    v\mid_{\R_k} \quad \text{ for } k=m+1,\dots,m+\ell \\ 
\end{cases}
$$
and similarly for $r^{(k)}$.

\begin{remark}
\label{rem:decomposition}
The Bellman equation $(I-\gamma P_\pi)v = r_\pi$ decomposes along the communicating
classes of $P_\pi$.
The restriction of $v$ to each recurrent class $\R_k$ depends only on
$P^{(k)}$ and $r^{(k)}$, while the restriction to each transient class $\T_j$
depends on $Q^{(j)}$, $r^{(j)}$, and on the values on downstream classes $v^{(j+1)}, \dots, v^{(m+\ell)}$.
This follows directly from the block structure \eqref{eq:matrix}.
\end{remark}

The following theorem shows that, in each recurrent class, the value function
is given by a PageRank problem as in
Theorem~\ref{thm:value_pagerank}, while in each transient class it can be expressed
as a PageRank vector associated with a time-reversed Doob transform and a
quasi-stationary distribution. 
\begin{theorem}
\label{thm:decomposition_pagerank}
Let $\pi$ be a fixed policy and $P_\pi$ be its induced transition matrix.
Decompose the state space $\mathcal S$ into transient and recurrent classes as in \eqref{eq:matrix}.

Then the value function $v_\pi$ can be computed by solving the following PageRank
problems on each communicating class as follows:
\begin{enumerate}
\item For each recurrent class $\R_k$, let $P^{(k)}$ be the restriction of $P_\pi$
to $\R_k$, and let $\mu^{(k)}$ be its stationary distribution.
The restriction $v^{(k)}$ is given by
$$
v^{(k)}(s) = \frac{\langle r^{(k)}\rangle_{\mu^{(k)}}}{1-\gamma}\frac{w^{(k)}(s)}{\mu^{(k)}(s)}, \qquad s \in \R_k,
$$
where $w^{(k)}$ is the PageRank vector of the time-reversed chain $(P^{(k)})^*$
with teleportation parameter $\gamma$ and restart distribution $u^{(k)}(s) = r^{(k)}(s) \mu^{(k)}(s) / \langle r^{(k)}\rangle_{\mu^{(k)}}$.

\item For each transient class $\T_j$, let $Q^{(j)}$ be the restriction of $P_\pi$
to $\T_j$, and let $\nu^{(j)}$ such that $\nu^{(j)} Q^{(j)} = \lambda^{(j)} \nu^{(j)}$ be its quasi-stationary
distribution.
The restriction $v^{(j)}$ is given by
$$
v^{(j)} = \frac{\langle \tilde r^{j}\rangle_{\nu^{(j)}}}{1-\gamma\lambda^{(j)}}\frac{w^{(j)}(s)}{\nu^{(j)}(s)}, \qquad s \in \T_j,
$$
where $w^{(j)}$ is the PageRank vector of the time-reversed Doob transform
$(\widetilde Q^{(j)})^*$ with teleportation parameter $\gamma \lambda^{(j)}$, and restart distribution 
\begin{align*}
    u^{(j)}(s) &= \frac{\tilde r^{(j)}(s)\nu^{(j)}(s)}{\langle \tilde r^{(j)} \rangle_{\nu^{(j)}}}, \\ \tilde r^{(j)} &:= r^{(j)} + \gamma\sum_{k=j+1}^{m} P_{\T_j\T_k} v^{(k)} + \gamma\sum_{k=1}^{\ell} P_{\T_j\R_k} v^{(k)}.
\end{align*}
\end{enumerate}
\end{theorem}

\noindent {\bf Proof:}
Since Theorem~\ref{thm:value_pagerank} applies directly to irreducible chains,
the value function on each recurrent class $\R_\ell$ can be computed independently
by restricting to the submatrix $P^{(\ell)}$.
Hence, it remains to study the value function on the transient classes.

From the structure of matrix \eqref{eq:matrix}, we see that, assuming the transient classes are sorted topologically, having computed the value function $v^{(k)}$ depends only on the value function of the downstream classes $v^{(k+1)},\dots, v^{(m+\ell)}$. We can include these dependency term on the right of Equation~\eqref{eq:bellman} as a new reward: 
\[
(I-\gamma Q^{(j)})v^{(j)} =
\]
\begin{equation}
\label{eq:Q_eqs}
r^{(j)} + \gamma\sum_{k=j+1}^{m} P_{\T_j\T_k} v^{(k)} + \gamma\sum_{k=1}^{\ell} P_{\T_j\R_k} v^{(k)} =: \tilde r^{(j)}.
\end{equation}

We are left to show that we can write Equation~\eqref{eq:Q_eqs} as a PageRank equation. Since $Q$ is substochastic, by using Lemma~\ref{lemma:trans} below concludes the proof.
\hfill $\Box$

To state and prove Lemma~\ref{lemma:trans} we need to overcome the main conceptual difficulty that $Q$ is substochastic and therefore does not admit a stationary distribution and time reversal.
The appropriate invariant object is the \emph{quasi-stationary distribution} \cite{seneta1966quasi}.
By Perron--Frobenius theory, since $Q$ is irreducible there exist strictly positive
left and right eigenvectors $\nu$ and $h$ such that
$$
\nu Q = \lambda \nu,
\qquad
Q h = \lambda h,
$$
where the vector $\nu$ represents the limiting distribution of the chain conditioned on
non-absorption, with $\lambda = \rho(Q) \in (0,1)$, the largest eigenvalue of $Q$. We normalize $h$ so that $\nu h = 1$.

Take the Doob transform \cite{Doob1957} \cite[Section~17.6]{levin2017markov} of $Q$:
\begin{equation}
\label{eq:doob}
\widetilde Q := \frac{1}{\lambda} H^{-1} Q H,
\end{equation}
where $H := \mathrm{diag}(h)$.
Then $\widetilde Q$ is a stochastic and irreducible transition matrix.
Its stationary distribution is
$$
\widetilde\nu(s) := \nu(s) h(s), \qquad s \in \T.
$$

The Doob transform therefore turns the transient dynamics into an ergodic Markov
chain, allowing us to define a meaningful time reversal.
Let $\widetilde N := \mathrm{diag}(\widetilde\nu)$.
We define the time-reversed transition matrix by
\begin{align}
\widetilde Q^*
&:= \widetilde N^{-1} \widetilde Q^\top \widetilde N = \frac{1}{\lambda} N^{-1} Q^\top N,
\label{eq:Qstar}
\end{align}
where $N := \mathrm{diag}(\nu)$.
The matrix $\widetilde Q^*$ is stochastic and has stationary distribution $\widetilde\nu$.
Conceptually, $\widetilde Q^*$ describes the reverse dynamics of the original chain
conditioned on non-absorption. Note that Equation~\eqref{eq:Qstar} does not depend on $h$.

\begin{lemma}
\label{lemma:trans}
Let $Q$ be an irreducible substochastic matrix with quasi-stationary
distribution $\nu$ and $\lambda=\rho(Q)$.
Let $b\neq 0$ be a non-negative vector, and let
$\gamma\in[0,1]$ satisfy $\gamma\lambda<1$.
Then the unique solution of
$$
(I-\gamma Q)x=b
$$
is given by
$$
x(s)=\frac{\langle b\rangle_\nu}{1-\gamma\lambda}\,\frac{w(s)}{\nu(s)}, \qquad \forall s \in \S,
$$
where $w$ is the PageRank vector of the time-reversed chain $\widetilde Q^*$ defined in
\eqref{eq:Qstar}, with teleportation parameter $\gamma\lambda$ and restart distribution
$$
u(s)=\frac{b(s)\nu(s)}{\langle b\rangle_\nu}.
$$
\end{lemma}


\noindent {\bf Proof:}
As in the proof of Theorem~\ref{thm:value_pagerank}, we take the transpose and right-multiplying by $N = \text{diag}(\nu)$
$$
x^\top N (I - \gamma \lambda \widetilde Q^*) = b^\top N.
$$
Multiplying both sides by $(1-\gamma\lambda) / \langle b\rangle_\nu$, we see that 
$w := (1-\gamma\lambda)x^\top N / \langle b\rangle_\nu$ coincides with the PageRank vector of $\widetilde Q^*$ with teleportation parameter $\gamma\lambda$ and restart distribution $u$.
\hfill $\Box$

\begin{remark}
   We note that the teleportation parameter is $\gamma\lambda$. Thus the effect of transience is to increase the teleport probability. 
\end{remark}

\subsection{Undiscounted MDPs}

Lemma~\ref{lemma:trans} extends naturally to undiscounted MDPs with
terminal (absorbing) states.
Let $\T$ and $\A$ be the sets of transient and absorbing states of $P_\pi$, respectively. 
By reordering states, we write the transition matrix as
$$
P_\pi =
\begin{pmatrix}
Q_\pi & P_{\T\A} \\
0 & I_{\A\A}
\end{pmatrix},
$$
where $Q_\pi$ is a substochastic matrix describing transitions among transient states
and $I_{\A\A}$ is the identity matrix on $\A$.
For ease of exposition, we assume that have only one transient class, i.e., $Q_\pi$ is irreducible. 

\begin{theorem}
\label{thm:abs}
Consider an undiscounted MDP with terminal states and non-negative reward vector $r$ on $T$, and $r(s)=0$ for all $s\in \A$.
Let $Q_\pi$ be an irreducible transition matrix on the transient states, and let
$\nu$ be its quasi-stationary distribution, $\lambda = \rho(Q_\pi)$.
Let $w$ be the PageRank vector of the time reversal $\widetilde Q^*$ (as defined in \eqref{eq:Qstar}) with teleportation
parameter $\lambda$ and restart distribution $u(s) = r(s)\nu(s) / \langle r_T\rangle_\nu$.
Then
$$
v_\pi(s) = \frac{\langle r_T\rangle_\nu}{1-\lambda}\frac{w(s)}{\nu(s)}, \quad \forall s \in \T,
\qquad
v_\pi(s) = 0, \quad \forall s \in \A .
$$
\end{theorem}

\noindent {\bf Proof:}
Assuming zero reward on $\A$ and setting $\gamma=1$, the teleportation parameter of
the associated PageRank problem becomes $\lambda<1$, so the formulation remains
well defined and Lemma ~\ref{lemma:trans} applies.
\hfill $\Box$


\section{Numerical Experiments}
\label{sec:experiments}

Theorem~\ref{thm:value_pagerank} reduces policy evaluation, in general, to two computations: the stationary distribution $\mu$ of the policy-induced chain and the PageRank vector $w$ of its time reversal $P_\pi^*$. In this section we apply these results to a class of controlled random walks, \emph{sticky random walks}, for which $\mu$ is available in closed form and the chain is reversible. Policy evaluation therefore requires only one PageRank solve.

\paragraph{Control system}
Let $G=(V,E)$ be a simple, undirected, connected graph. A walker at node $x$
stays at $x$ with probability $\alpha_x$ and moves to a uniformly chosen
neighbour with probability $1-\alpha_x$. The \emph{stickiness} $\alpha_x$ is the
action, available only on set $\mathcal C\subseteq V$ of
\emph{control nodes}. The action space is $\mathscr{A}=\{0.1,0.2,\ldots,0.9\}$ at control nodes and $\{0\}$ elsewhere, so the walker
there is an ordinary random walker on $G$. The chain is reversible for any
$\alpha$, with
\begin{equation}
\label{eq:sticky_mu}
\mu(x)\;\propto\;\frac{d_x}{1-\alpha_x} ,
\end{equation}  
where $d_x$ is the degree of node $x$. The one-step reward is $r_x = \beta_x - c(\alpha_x)$,
where $\beta_x$ is a state reward and $c(\alpha_x)=\kappa \alpha_x/(1-\alpha_x)$ with $\kappa > 0$ is the action cost. We take $\beta$ carried by the
control nodes, $\beta_x=0$ for $x\notin\mathcal C$, so that the reward vector is supported on $\mathcal{C}$. 


\paragraph{Algorithms}
The Bellman and PageRank residual are defined by $\text{res}_v := r + \gamma P_\pi - v$ and $\text{res}_w := \gamma w P^* + (1-\gamma)u - w$, respectively.
Then, \eqref{eq:vf_pr} of Theorem~\ref{thm:value_pagerank} gives
\begin{equation}
    \mu(s)|\text{res}_v(s)| \propto |\text{res}_w(s)|.
    \label{eq:res_id}
\end{equation}
Thus, the PageRank residual is the stationary-distribution-weighted Bellman
residual. Equation~\eqref{eq:res_id} transports any residual-based selection
rule for PageRank into a priority rule for the Bellman equation. Classical
prioritized sweeping ranks states by $|\text{res}_v|$
\cite{moore1993prioritized,andre1997generalized}, whereas the RLGL heuristics
apply their priorities in the PageRank coordinates
\cite{RLGL,RLGL_VAR}.

We compare two Bellman solvers (their names are printed in italic) with two PageRank solvers:
\medskip

\noindent\textit{\textsc{Gauss--Seidel}}:
Round-robin state updates applied directly to the Bellman system
\cite[Chapter~6]{puterman2014markov}.

\smallskip
\noindent\textit{\textsc{Prioritized Sweeping}}:
Residual-driven updates that always select a state maximizing
$|\operatorname{res}_v(s)|$
\cite{moore1993prioritized,andre1997generalized}.

\smallskip
\noindent\textsc{RLGL--GSD}:
The Red--Light--Green--Light coordinate method \cite{RLGL} with
the \textsc{GSD} priority rule \cite{RLGL_VAR}, which selects a state
maximizing $\sqrt{\mu(s)}\,|\operatorname{res}_v(s)|$. The batch version fixes the order of updates for a batch of steps. 

\smallskip
\noindent\textsc{FwdPushSOR}:
Forward push with successive over-relaxation for personalized PageRank
\cite{chen2023accelerating}.
\medskip

\paragraph{Metrics and implementation}
We report the $\ell_1$ relative Bellman residual $\|\text{res}_v\|_1 / \Vert r\Vert_1$, wall-clock runtimes, and the \emph{normalized iterations} (cumulative number of graph edges processed divided by the total number of edges). All methods use the same stopping tolerance $10^{-10}$, $\hat v_0 = r$, and results are averaged over $10$ independent runs. Experiments are run with $\gamma = 0.9$, $\kappa=0.1$ on preferential-attachment graphs (PA), Erdős–Rényi (ER) \cite{REMCO}, 2D grid, and a chain of cliques, with $n=10^5$ nodes. For ER and chain of clique graphs, we sample $\beta_x\overset{\mathrm{iid}}{\sim}\mathrm{Unif}[0,1]$. For PA and 2D grid graphs, we use $\beta_x=1/(d(x,x^*)+1)$, where $d(x,x^*)$ is the graph distance between $x$ and the target node $x^*$, which is the central node in 2D grid and the maximum-degree node in PA. The control set $\mathcal C$ is sampled uniformly among subsets of the prescribed size, and initial actions $\alpha_x$, $x\in\mathcal C$, are sampled independently and uniformly from $\mathscr{A}$.

\begin{table*}[t]
\centering
\scriptsize
\setlength{\tabcolsep}{3pt}
\begin{tabular}{@{}lcclrrrrrrrr@{}}
\toprule
&&&& \multicolumn{4}{c}{Policy evaluation}
& \multicolumn{4}{c}{Policy iteration} \\
\cmidrule(lr){5-8}\cmidrule(l){9-12}
&&&& \multicolumn{2}{c}{Dense: $|\mathcal C|=n$}
& \multicolumn{2}{c}{Localized: $|\mathcal C|=10$}
& \multicolumn{2}{c}{Dense: $|\mathcal C|=n$}
& \multicolumn{2}{c}{Localized: $|\mathcal C|=10$} \\
\cmidrule(lr){5-6}\cmidrule(lr){7-8}
\cmidrule(lr){9-10}\cmidrule(l){11-12}
Instance & $\bar d$ & $d_{\max}$ & Method
& Iter& Time (ms) & Iter & Time (ms)
& Iter & Time (s)  & Iter & Time (s) \\
\midrule

2D grid & 3.99 & 4
& \textit{\textsc{Gauss--Seidel}}
& 185 & \textbf{111} & 139 & 79.6
& 701 & \textbf{0.406} & 437 & 0.255 \\
& &
& \textit{\textsc{Prioritized Sweeping}}
& 160 & 7686 & 23.1 & 1204
& 602 & 24.1 & 55.2 & 2.79 \\
& &
& \textsc{RLGL--GSD}
& \textbf{96.7} & 6425 & 23.1 & 1254
& 520 & 22.9 & 54.6 & 2.84 \\
& &
& \textsc{RLGL--GSD} (batch)
& 96.8 & 509 & 38.2 & 156
& 596 & 2.28 & 288 & 0.849 \\
& &
& \textsc{FwdPushSOR}
& 119 & 565 & \textbf{13.0} & \textbf{48.3}
& \textbf{393} & 1.31 & \textbf{37.8} & \textbf{0.116} \\
\midrule

ER & 13.8 & 33
& \textit{\textsc{Gauss--Seidel}}
& 172 & \textbf{393} & 120 & \textbf{280}
& 647 & \textbf{1.56} & 388 & \textbf{0.916} \\
& &
& \textit{\textsc{Prioritized Sweeping}}
& 160 & 12354 & 112 & 7330
& 632 & 46.1 & 344 & 23.3 \\
& &
& \textsc{RLGL--GSD}
& \textbf{50.8} & 6403 & 120 & 8842
& 570 & 46.5 & 352 & 26.5 \\
& &
& \textsc{RLGL--GSD} (batch)
& \textbf{50.8} & 677 & 120 & 1699
& 571 & 7.69 & 352 & 4.73 \\
& &
& \textsc{FwdPushSOR}
& 118 & 1033 & \textbf{72.2} & 1045
& \textbf{391} & 4.41 & \textbf{227} & 3.34 \\
\midrule

PA & 6 & 1010
& \textit{\textsc{Gauss--Seidel}}
& 160 & \textbf{210} & 120 & \textbf{155}
& 646 & \textbf{0.807} & 379 & \textbf{0.494} \\
& &
& \textit{\textsc{Prioritized Sweeping}}
& 160 & 9010 & 112 & 4660
& 640 & 31.2 & 335 & 13.8 \\
& &
& \textsc{RLGL--GSD}
& \textbf{83.8} & 6462 & 149 & 6294
& 709 & 33.4 & 435 & 18.0 \\
& &
& \textsc{RLGL--GSD} (batch)
& 104 & 755 & 129 & 744
& 602 & 3.77 & 387 & 2.50 \\
& &
& \textsc{FwdPushSOR}
& 106 & 534 & \textbf{69.1} & 530
& \textbf{392} & 2.27 & \textbf{234} & 1.73 \\
\midrule

Clique chain & 63 & 64
& \textit{\textsc{Gauss--Seidel}}
& 185 & \textbf{1241} & 139 & 1415
& 653 & \textbf{4.61} & 415 & 4.00 \\
& &
& \textit{\textsc{Prioritized Sweeping}}
& 160 & 26526 & 3.35 & 1632
& 593 & 92.4 & 11.9 & 3.97 \\
& &
& \textsc{RLGL--GSD}
& \textbf{83.8} & 15565 & 3.35 & 1589
& 515 & 88.7 & 11.9 & 4.16 \\
& &
& \textsc{RLGL--GSD} (batch)
& \textbf{83.8} & 2799 & 4.15 & 79.1
& 515 & 17.3 & 28.2 & 0.518 \\
& &
& \textsc{FwdPushSOR}
& 117 & 2248 & \textbf{2.10} & \textbf{65.9}
& \textbf{369} & 7.54 & \textbf{10.0} & \textbf{0.221} \\
\bottomrule
\end{tabular}
\caption{Policy-evaluation and policy-iteration performance. $\bar d$ and
$d_{\max}$ are the average and maximum degree of the graph, which excludes the
self-loop the sticky walk adds at every state. Iter refers to normalized
iterations; Time refers to the wall-clock time.}
\label{tab:eval_and_pi}
\end{table*}
\paragraph{Results}
As we see in both Figure~\ref{fig:conv_eval} and Table~\ref{tab:eval_and_pi}, in terms of normalized iterations, PageRank-based methods consistently and significantly outperform the Bellman-based methods. In terms of wall-clock time, in Table~\ref{tab:eval_and_pi} we observe a sharp difference between the dense and the localized rewards. With dense rewards, \textit{\textsc{Gauss--Seidel}} is always the fastest in our experiments. With localized rewards, the results are more nuanced. When graphs are structured (2D-grid and chain of cliques), the {\sc FwPushSOR} is convincingly the fastest; indeed, localized rewards induce a sparse restart distribution, which is close to Personalized PageRank, for which {\sc FwPushSOR} is designed. When the graphs are good expanders (ER and PA), \textit{\textsc{Gauss--Seidel}} is still the fastest. In these graphs, the residual quickly spreads to a large number of states, suggesting that not only localized rewards, but also localization of the residual defines which algorithm is the fastest.

\begin{figure*}[t]
\centering
\includegraphics[width=\textwidth]{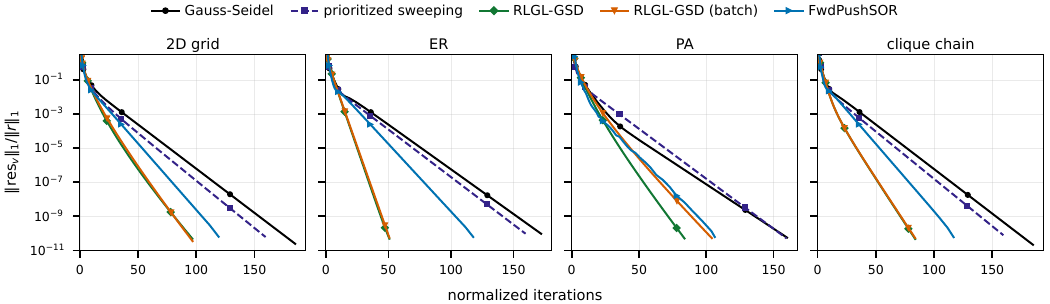}
\end{figure*}

\begin{figure*}[t]
\centering
\includegraphics[width=\textwidth]{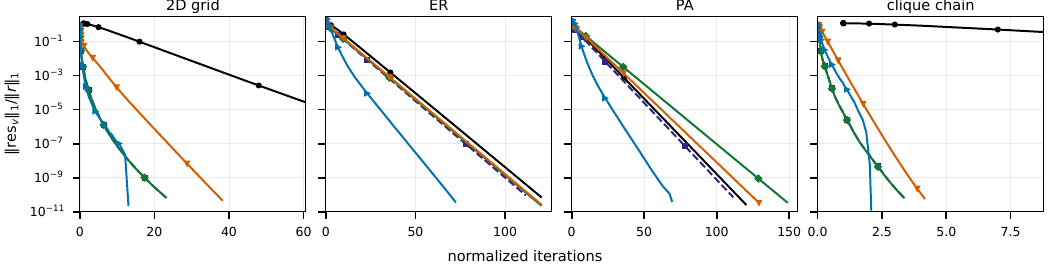}
\caption{Policy evaluation in the dense regime (top row) $|\mathcal C|=n$ and localized regime (bottom row): relative
$\ell_1$ Bellman residual against normalized iterations.}
\label{fig:conv_eval}
\end{figure*}

\section{Conclusion}
\label{sec:conclusion}
We showed that policy evaluation admits a decomposition into PageRank computations, one for each communicating class of the induced Markov chain.  Beyond its theoretical interest, this viewpoint has practical implications, because a broad class of algorithms developed for PageRank--including Monte Carlo and coordinate-based methods--can be directly repurposed as Bellman solvers, and may introduce new approaches to policy improvement.
We illustrated this connection in a reversible setting, where policy evaluation reduces to a single PageRank computation. 


With current implementation, we recommend the PageRank-based methods for localized rewards and structured graphs. Structured transitions occur in many systems such as queueing networks and systems with spatial structure. Such systems can directly benefit from the algorithms proposed in this work. Moreover, our methods consistently ourperform Bellman solvers in terms of the number of normalized iterations. Possibly, with other implementation, these algorithms have the potential for shorter runtime as well. Finally, it is natural to ask whether this perspective can yield tangible improvements in large-scale or deep reinforcement learning, where policy evaluation remains a computational bottleneck.

\bibliographystyle{plain}
\bibliography{biblio}

\end{document}